\documentclass[reqno,b5paper]{amsart}
\usepackage{amsmath}
\usepackage{amssymb}
\usepackage{amsthm}
\usepackage{enumerate}
\usepackage[mathscr]{eucal}

\newtheorem{thm}{Theorem}[section]

\newtheorem{lem}{Lemma}[section]
\newtheorem{prop}{Proposition}[section]

\newtheorem{ex}{Example}[section]
\newtheorem{rem}{Remark}[section]

\numberwithin{equation}{section}
\begin{document}

\title[Approximation theorems connected with Dunkl operator]{Approximation theorems connected with differential-difference operator}

\author[C. Abdelkefi \and S. Chabchoub]{Chokri Abdelkefi* \and Safa Chabchoub**
}
\newcommand{\acr}{\newline\indent}
\address{\llap{*\,}
Department of Mathematics\acr Preparatory Institute of Engineer
Studies of Tunis \acr 1089 Monfleury Tunis, University of Tunis\acr
Tunisia} \email{chokri.abdelkefi@yahoo.fr}
\address{\llap{**\,} Department of Mathematics\acr Faculty of Sciences of Tunis\acr 1060
Tunis, University of Tunis El Manar\acr Tunisia}
\email{safachabchoub@yahoo.com}

\thanks{This work was completed with the support of the DGRST research project LR11ES11,
University of Tunis El Manar, Tunisia.}

\begin{abstract}
In the present paper, we propose to give an extension to the context
of Dunkl theory of the notion of translation and in connection with
this a corresponding extension of Taylor's formula. More precisely,
we prove some properties and estimates of the integral remainder in
the generalized Taylor formula associated to the Dunkl operator on
the real line and we describe the Besov-type spaces for which the
remainder has a given order.
\end{abstract}
\maketitle
\section{Introduction}
\label{intro}   \label{intro}  Delsartes gave in \cite{Del1,Del2} a
certain extension of the notion of translation and
 in connection with this a corresponding extension of Taylor's
 formula. Generalized translation have later been considered from
 various points of view by many authors (see Levitan \cite{Lev}, Bochner \cite{Boch1,Boch2}).
 L\"{o}fstr\"{o}m and Peetre in \cite{L.P} estimated
 the remainder in the generalized Taylor's formula and they described the space of
 functions for which the remainder has a given order.\\

 Our aim in this paper is to extend the results obtained in
\cite{L.P} to the context of Dunkl theory. More precisely, we prove
some properties and estimates of the integral remainder of order $k$
associated to the Dunkl operator on the real line and we establish
the coincidence between two characterizations of Besov-type spaces
related to this integral
remainder.\\

 For a real parameter $\alpha > -\frac{1}{2}$, the
Dunkl operator on the real line denoted by $\Lambda_{\alpha}$, is a
differential-difference operator introduced in 1989 by C. Dunkl in
\cite{dun}. This operator is associated with the reflection group $
\mathbb{Z}_{2}$ on $\mathbb{R}$ and is given by
$$ \Lambda_\alpha f(x) = \frac{df}{dx} (x) + \frac{2\alpha+1}{x}
 \Big[\frac{f(x)-f(-x)}{2}\Big],\; f \in \mathcal{C}^{1}(\mathbb{R}).$$
The Dunkl operator can be considered as a perturbation of the usual
derivative by reflection part. This operator plays a major role in
the study of quantum harmonic oscillators governed by Wigner's
commutation rules (see \cite{rose}). The Dunkl kernel $E_{\alpha}$
related to $\Lambda_{\alpha}$ is used to define the Dunkl transform
which enjoys properties similar to those of the classical Fourier
transform. The Dunkl kernel $E_{\alpha}$
 satisfies a product formula (see \cite{ro1}). This allows us to define
the Dunkl translation $\tau_{x}$, $x\in\mathbb{R}$ (see next
section). If the parameter $\alpha = -\frac{1}{2}$, then the
operator $\Lambda_\alpha$ reduces to the differential operator
$\frac{df}{dx}$. Therefore Dunkl analysis can be viewed as a
generalization of the classical Fourier analysis on $\mathbb{R}$.

In 2003, the classical Taylor formula with integral remainder was
extended in \cite{mou} to the one dimensional Dunkl operator
$\Lambda_{\alpha}$ :\\ For $k=1,2,...,$ $f \in
\mathcal{E}(\mathbb{R})$ and $ a \in \mathbb{R}$, we have
\begin{align*} \tau _x(f)(a) = \sum_{p=0}^{k-1} b_p(x) \Lambda_\alpha^p f(a)
 +  R_k(x,f)(a),\quad x \in \mathbb{\mathbb{R}}\backslash\{0\} ,\end{align*}
with $R_k(x,f)(a)$ is the integral remainder of order $k$ given by
 \begin{align*} \displaystyle R_k(x,f)(a)= \int_{-|x|}^{|x|} \Theta_{k-1} (x,y)
  \tau_y (\Lambda_\alpha^{k} f)(a) A_\alpha(y) dy,\end{align*}
 where $\mathcal{E}(\mathbb{R})$ is the space of infinitely
differentiable functions on $\mathbb{R}$ and $ (\Theta_{p})_{p\in
\mathbb{N}}$, $(b_p)_{p\in \mathbb{N}}$ are two sequences of
 functions constructed inductively from the function $A_\alpha$ defined on $\mathbb{R}$ by $A_\alpha(x)=
 |x|^{2\alpha+1}$ (see next section).

There are many ways to define the Besov spaces (see
\cite{An,Bes,Pe}) and the Besov-Dunkl spaces (see
\cite{ab1,ab2,ab3,ab4}). It is well known that Besov spaces can be
described by means of differences using the modulus of continuity of
functions. These spaces defined by the modulus of smoothness occur
more naturally in many areas of analysis including approximation
theory.

In this paper we define the following weighted function spaces:
\\Let $0<\beta <1$, $1 \leq p < +\infty $, $ 1 \leq q \leq +\infty$
and $k$ a positive integer $(k=1, 2,...)$.
\\$\bullet$ We denote by $L^p(\mu_\alpha)$ the space of complex-valued functions
$f$, measurable on $\mathbb{R}$ such that
$$\|f\|_{p,\alpha} = \left(\int_{ \mathbb{R}}|f(x)|^p
d\mu_\alpha(x) \right)^{1/p} < + \infty,$$ where $\mu_\alpha$ is a
weighted Lebesgue measure associated to the Dunkl operator given by
\begin{align*}d\mu_\alpha(x) = \frac{A_\alpha(x)}{2^{\alpha +1}\Gamma(\alpha
+1)}dx,\end{align*} with $A_\alpha$ is the function defined on
$\mathbb{R}$ by
$$A_\alpha(x) = |x|^{2\alpha+1},\quad x \in \mathbb{R}.$$
$\bullet$ The Besov-Dunkl space of order $k$ denoted
 by $\mathcal{B}^k\mathcal{D}_{p,q}^{\beta,\alpha}$
is the subspace of functions $f$ in $\mathcal{E}(\mathbb{R})\cap
L^p(\mu_\alpha)$ such that $\Lambda_\alpha^{k-1}f \in
L^p(\mu_\alpha)$ and satisfying
\begin{align*} \int_0^{+\infty} \Big(\frac{\omega_{p,\alpha}^k(x,f)}{x^{\beta+k-1}}\Big)^q
 \frac{dx}{x} < +\infty \quad &if \quad  q < +\infty \\
\mbox{and} \qquad
    \sup_{x > 0} \frac{\omega_{p,\alpha}^k(x,f)}{x^{\beta+k-1}} < +\infty  \quad &if \quad   q =
  +\infty,\end{align*}
with $\omega_{p,\alpha}^k(x,f) =\displaystyle \sup_{|y| \leq x} \|
R_{k-1}(y, f)-b_{k-1}(y)\Lambda_\alpha^{k-1}f\|_{p,\alpha},$ where
we put $\Lambda_\alpha^{0}f=f$ and $R_{0}(x, f)=\tau _x(f),$ for
$k=1.$
 \\$\bullet$ Put $\mathcal{D}_{p,\alpha}^k$ the subspace of functions $f$ in
$\mathcal{E}(\mathbb{R})\cap L^p(\mu_\alpha)$ such that
  $ \Lambda_\alpha^k f$ are in $L^p(\mu_\alpha)$.
We consider the subspace
$\mathcal{K}^k\mathcal{D}_{p,q}^{\beta,\alpha}$ of functions
 $f \in \mathcal{D}_{p,\alpha}^{k-1}+\mathcal{D}_{p,\alpha}^{k} $ satisfying
 \begin{align*}
 \int_0^{+\infty} \Big(\frac{K_{p,\alpha}^k(x,f)}{x^{\beta}}\Big)^q \frac{dx}{x} < +\infty \quad &if \quad  q <
 +\infty\\
\mbox{and}\qquad
   \sup_{x > 0} \frac{K_{p,\alpha}^k(x,f)}{x^{\beta}} < +\infty \quad &if \quad    q = +\infty,
\end{align*} where $K_{p,\alpha}^k $  is the Peetre K-functional given by
$$K_{p,\alpha}^k(x,f)= \inf_{f=f_0+f_1}\Big\{\| \Lambda_\alpha^{k-1} f_0 \|_{p,\alpha}+ x \| \Lambda_\alpha^k f_1
\|_{p,\alpha},\,
 f_0 \in \mathcal{D}_{p,\alpha}^{k-1},\, f_1 \in
 \mathcal{D}_{p,\alpha}^k\Big\}.$$

The contents of the present paper are as follows. \\In section 2, we
collect some basic definitions and results about harmonic analysis
associated with the Dunkl operator $\Lambda_\alpha$. \\
In section 3, we give some properties and estimates of the integral
remainder of order $k$. Finally, we establish that
$$\mathcal{B}^k\mathcal{D}_{p,q}^{\beta,\alpha}=\mathcal{K}^k\mathcal{D}_{p,q}^{\beta,\alpha}.$$

Along this paper, we use $c$ to represent a suitable positive
constant which is not necessarily the same in each occurrence.
\section{Preliminaries}
\label{sec:1}  In this section, we recall some notations and results
in Dunkl theory on $\mathbb{R}$ and we refer for more details to \cite{am,dun,ro1}.\\

 For $\lambda \in \mathbb{C}$, the initial problem
$$\Lambda_\alpha(f)(x) = \lambda f(x),\quad f(0) = 1,\quad x \in \mathbb{R},$$
has a unique solution $E_\alpha(\lambda .)$ called Dunkl kernel
given by
$$E_\alpha(\lambda x) = j_\alpha(i\lambda x) + \frac{\lambda x}
{2(\alpha+1)} j_{\alpha+1} (i\lambda x),\quad x \in \mathbb{R},$$
where $j_\alpha$ is the normalized Bessel function of the first kind
and order $\alpha.$

 The Dunkl kernel $E_\alpha$ satisfies the following
product formula
$$E_\alpha(ixt) E_\alpha(iyt) = \int_{\mathbb{R}} E_\alpha(itz)
d\gamma_{x,y}(z),\quad x, y, t \in \mathbb{R}, $$ where
$\gamma_{x,y}$ is a signed measure on $\mathbb{R}$ with compact
support.

For $x, y \in \mathbb{R}$ and $f$ a continuous function on
$\mathbb{R}$, the Dunkl translation operator $\tau_x$ is given by
 $$\tau_x(f)(y) =\int_{\mathbb{R}} f(z) d\gamma_{x,y}(z).$$.
 \\The Dunkl translation
operator satisfies the following properties :
\begin{enumerate}
\item $\tau_x$ is a continuous linear operator from
$\mathcal{E}( \mathbb{R})$ into itself.
\item For all $f \in
\mathcal{E}(\mathbb{R})$,  we have \begin{align*}\tau_x(f)(y) =
\tau_y(f)(x)\quad\mbox{and}\quad \tau_0(f)(x) = f(x)\end{align*}
\begin{align}\tau_x \,o\, \tau_y = \tau_y\,o\,\tau_x\quad\mbox{and}\quad
\Lambda_\alpha \,o\,\tau_x = \tau_x \,o\,\Lambda_\alpha .\end{align}
\item For all $x \in \mathbb{R}$, the
operator $\tau_x$ extends to $L^p(\mu_\alpha),\; p \geq 1$ and we
have for $f$ in $L^p(\mu_\alpha)$
\begin{align}
\|\tau_x(f)\|_{p,\alpha} \leq \sqrt{2} \|f\|_{p,\alpha}.
\end{align}
\end{enumerate}
It has been shown in \cite{mou}, the following generalized Taylor
formula with integral remainder:
\begin{prop} For $k=1,2,...,$ $f \in
\mathcal{E}(\mathbb{R})$ and $ a \in \mathbb{R}$, we have
\begin{eqnarray} \tau _x f(a) = \sum_{p=0}^{k-1} b_p(x) \Lambda_\alpha^p f(a) +  R_k(x,f)(a),\quad x \in \mathbb{R}\backslash\{0\} ,\end{eqnarray}
with $ R_k(x,f)(a)$ is the integral remainder of order $k$ given by
 \begin{eqnarray} \displaystyle R_k(x,f)(a)= \int_{-|x|}^{|x|} \Theta_{k-1} (x,y) \tau_y (\Lambda_\alpha^{k} f)(a) A_\alpha(y) dy,\end{eqnarray}
where\begin{enumerate}
\item[(i)] $\displaystyle b_{2m}(x)= \frac{1}{(\alpha+1)_m m!}  \Big(\frac{x}{2} \Big)^{2m}$ and $\;\displaystyle
  b_{2m+1}(x)= \frac{1}{(\alpha+1)_{m+1} m!}  \Big(\frac{x}{2}
  \Big)^{2m+1}$, for $\;m\in \mathbb{N}.$
 \item[(ii)] $ \Theta_{k-1}(x,y) = u_{k-1}(x,y) + v_{k-1}(x,y)\;$
 with
   $\;\displaystyle u_0(x,y)= \frac{sgn(x)}{2 A_\alpha(x)} , $\\ $\,\displaystyle v_0(x,y)= \frac{sgn(y)}{2
 A_\alpha(y)},$
 $\displaystyle u_k(x,y)= \int_{|y|}^{|x|} v_{k-1}(x,z) dz\;$ and \\$\;\displaystyle
 v_k(x,y)= \frac{sgn(y)}{ A_\alpha(y)}\int_{|y|}^{|x|} u_{k-1}(x,z)A_\alpha(z)
   dz.$
   \end{enumerate}
   \end{prop}
   According to (\cite{ro3}, Lemma 2.2), the Dunkl operator
   $\Lambda_\alpha$ have the following regularity properties:
\begin{align} \Lambda_\alpha \;\mbox{leaves}\;\, \mathcal{C}_c^\infty(\mathbb{R})\;
 \mbox{and} \;\mbox {the\, Schwartz\, space}\; \mathcal{S}(\mathbb{R})\; \mbox{invariant}. \end{align}
\section{Characterizations of Besov-Dunkl spaces of order $k$}
In this section, we start with the proof of some properties and
estimates of the integral remainder in the generalized Taylor
formula.
\begin{rem}
 Let $k=1,2,...,$ $f \in \mathcal{E}(\mathbb{R})$ and $ x \in
 \mathbb{R}\backslash\{0\}$.
 \begin{enumerate}
\item From Proposition 2.1, we have
  \begin{align}
R_k(x, f)& = \tau_x(f)- f-b_1(x)\Lambda_\alpha f...- b_{k-1}(x)\Lambda_\alpha^{k-1}f \nonumber  \\
  &=R_{k-1}(x, f)-b_{k-1}(x)\Lambda_\alpha^{k-1} f,
  \end{align} where we put for $k=1$, $R_{0}(x, f) = \tau_x(f).$
  Observe that $$R_1(x, f)=R_{0}(x, f)-b_{0}(x)\Lambda_\alpha^{0} f=\tau_x(f)-
  f.$$
  \item According to (\cite{mou}, p.352) and Proposition 2.1, (i), we have
 \begin{align}
 \displaystyle \int_{-|x|}^{|x|} |\Theta_{k-1} (x,y)|  A_\alpha(y) dy &\leq b_k(|x|)+|x| b_{k-1}(|x|)\nonumber \\
  &\leq c\, |x|^{k}.
\end{align}
\item Note that the function $y\longmapsto
\tau_y(f)-f$ is continuous on $\mathbb{R}$ (see \cite{mou.T}, Lemma
1, (ii)), which implies that the same is true for the function
$y\longmapsto R_k(y, f).$
 \end{enumerate}
\end{rem}
\begin{lem}
 Let $k=1,2,...,$ then there exists a constant $c>0$ such that for all $f \in \mathcal{E}(\mathbb{R})$ satisfying
 $\Lambda_\alpha^{k-1}f \in L^p(\mu_\alpha) $, we have
 \begin{align}
  \| R_{k-1}(x,f)\|_{p,\alpha} \leq c \,|x|^{k-1} \| \Lambda_\alpha^{k-1} f \|_{p,\alpha},\quad x \in \mathbb{R}
  \backslash\{0\}.
 \end{align}
 \end{lem}
\begin{proof}
 Let $k=1,2,...,$ $f \in \mathcal{E}(\mathbb{R})$ such that $\Lambda_\alpha^{k-1}f \in L^p(\mu_\alpha) $ and $x \in
 \mathbb{R}\backslash\{0\}$. For $k=1$, by (2.2), it's clear that $\|R_{0}(x, f)\| =\| \tau_x(f)\|_{p,\alpha}\leq c\,\|f\|_{p,\alpha}.$
Using Minkowski's inequality for integrals, (2.2) and (2.4), we have
for $k\geq2$
 \begin{align*}
  \| R_{k-1}(x,f)\|_{p,\alpha}  &\leq \int_{-|x|}^{|x|}| \Theta_{k-2} (x,y)| \
   \| \tau_y ( \Lambda_\alpha^{k-1} f)\|_{p,\alpha} A_\alpha(y) dy\\
   &\leq c \;  \|   \Lambda_\alpha^{k-1} f \|_{p,\alpha} \int_{-|x|}^{|x|} |\Theta_{k-2} (x,y)| A_\alpha(y) dy.
 \end{align*}
From (3.2),  we deduce our result.
\end{proof}
\begin{rem} Let $k=1,2,...,$ $f\in\mathcal{E}(\mathbb{R})$ such that $\Lambda_\alpha^{k-1}f \in L^p(\mu_\alpha).$
Then for $x \in \mathbb{R}\backslash\{0\},$ we have by (3.1), (3.3)
and Proposition 2.1,
 \begin{align}
 \|R_k(x,f)\|_{p,\alpha} &\leq
\|R_{k-1}(x,f)\|_{p,\alpha}+
    \|b_{k-1}(x) \Lambda_\alpha^{k-1} f\|_{p,\alpha} \nonumber
     \\&\leq c \,|x|^{k-1} \|\Lambda_\alpha^{k-1} f \|_{p,\alpha}.
 \end{align}
\end{rem}
\begin{lem}
For $x \in \mathbb{R}\backslash\{0\}$ and $p\in \mathbb{N}$, we have
\begin{align}
 \int_{-|x|}^{|x|} \Theta_0 (x,y)  b_p(y)  A_\alpha(y) dy = b_{p+1}(x).
\end{align}
\end{lem}
\begin{proof}
 Let $x\in \mathbb{R}\backslash\{0\} $. Using Proposition 2.1, we
 have:\\
$\bullet$ If $p=2m,$ $m\in \mathbb{
N},$\\$\displaystyle\int_{-|x|}^{|x|} \Theta_0 (x,y) b_{2m}(y)
A_\alpha(y)
 dy$
  \begin{align*}
    & = \int_{-|x|}^{|x|} \frac{sgn(x)|y|^{2\alpha+1}}{2|x|^{2\alpha+1}} b_{2m}(y) dy + \int_{-|x|}^{|x|} \frac{sgn(y)}{2} b_{2m}(y)  dy \\
   &= \frac{x}{2^{2m}|x|^{2\alpha+2}(\alpha+1)_m m!}\int_{0}^{|x|} y^{2\alpha+2m+1} dy\\
   &=\frac{x}{2^{2m}(\alpha+1)_m m!} \frac{|x|^{2m}}{2 (\alpha +m+1)} \\
   &=b_{2m+1}(x).
  \end{align*}
$\bullet$ If $p =2m+1,$ $m\in \mathbb{N},$ we get\\
$\displaystyle\int_{-|x|}^{|x|}  \Theta_0 (x,y)  b_{2m+1}(y)
A_\alpha(y) dy$
  \begin{align*}
      & = \int_{-|x|}^{|x|} \frac{sgn(x)|y|^{2\alpha+1}}{2|x|^{2\alpha+1}} b_{2m+1}(y) dy + \int_{-|x|}^{|x|} \frac{sgn(y)}{2} b_{2m+1}(y)  dy \\
   &=  \frac{1}{2^{2m+1}(\alpha+1)_{m+1} m!}\int_{0}^{|x|} y^{2m+1} dy\\
   &=\frac{1}{2^{2m+1}(\alpha+1)_{m+1} m!} \frac{|x|^{2m+2}}{2 (m+1)} \\
   &= b_{2m+2}(x).
  \end{align*}
Hence the Lemma is proved.
\end{proof}
\begin{lem}
  Let $k=1,2,...,$ $f \in \mathcal{E}(\mathbb{R})$, $x \in \mathbb{R}\backslash\{0\}$ and $a \in
  \mathbb{R}$. Then we have,
 \begin{align}
R_k(x, f)(a) =\int_{-|x|}^{|x|} \Theta_0(x,y)
R_{k-1}(y,\Lambda_\alpha f)(a) A_\alpha(y) dy.
\end{align}
\end{lem}
\begin{proof}
Let $k=1,2,...,$ $f \in \mathcal{E}(\mathbb{R})$, $x \in
\mathbb{R}\backslash\{0\}$ and $a \in \mathbb{R}$. We have from
(2.3), (2.4) and the fact that $R_0(y,\Lambda_\alpha f)(a)=
\tau_y(\Lambda_\alpha  f),$
$$  R_1(x,f)(a) = (\tau_x(f)-f)(a)=\int_{-|x|}^{|x|} \Theta_0(x,y) \tau_y(\Lambda_\alpha f)(a) A_\alpha(y)
dy,
$$ hence the property (3.6) is true for $k=1.$
  \\ Suppose that
  $$ R_k(x,f)(a) =\int_{-|x|}^{|x|} \Theta_0(x,y) R_{k-1}(y,\Lambda_\alpha f)(a) A_\alpha(y) dy, $$
   then by (3.1) and (3.5), we get\\$\displaystyle\int_{-|x|}^{|x|} \Theta_0(x,y) R_k(y,\Lambda_\alpha f)(a) A_\alpha(y) dy$
 \begin{align*}
   &=\int_{-|x|}^{|x|} \Theta_0(x,y)R_{k-1}(y,\Lambda_\alpha f)(a)A_\alpha(y) dy \\
  &- \int_{-|x|}^{|x|} \Theta_0(x,y) b_{k-1}(y)\Lambda_\alpha^{k} f(a)A_\alpha(y) dy\\
  &= R_k(x,f)(a) - \Lambda_\alpha^k f(a)\int_{-|x|}^{|x|} \Theta_0(x,y) b_{k-1}(y)A_\alpha(y) dy \\
  &= R_k(x,f)(a) - b_k(x)\Lambda_\alpha^k f(a)\\
  &= R_{k+1}(x, f)(a).
  \end{align*}
 By induction, we deduce our result.
\end{proof}
\begin{lem}
Let $k=1,2,...,$ $f \in \mathcal{E}(\mathbb{R})$, $x \in
\mathbb{R}\backslash\{0\}$ and $a \in
  \mathbb{R}$. We denote by
 $$ \displaystyle I_1(x,f)(a) = \int_{-|x|}^{|x|} \Theta_0 (x,y)
  \tau_y (f)(a) A_\alpha(y) dy,$$
  and for  $k \geq 2$
  $$\displaystyle I_k(x,f)(a) = \int_{-|x|}^{|x|} \Theta_0 (x,y)
  I_{k-1}(y,f) (a) A_\alpha(y) dy.$$
Then, we have
\begin{align}\Lambda_\alpha^{k+1}
\big(I_k(x,f)\big)(a)&=\Lambda_\alpha^{k} \big(I_k(x,\Lambda_\alpha
f)\big)(a) \\\mbox{and}\qquad \Lambda_\alpha^k I_k(x,f)(a) &=
R_k(x,f)(a).\end{align}
\end{lem}
\begin{proof} Let $k=1,2,...,$ $f \in \mathcal{E}(\mathbb{R})$, $x \in
\mathbb{R}\backslash\{0\}$ and $a \in
  \mathbb{R}$.\\
$\bullet$ Using (2.1), we have \begin{eqnarray*} \Lambda_\alpha^{2}
\big(I_1(x,f)\big)(a)&= &  \int_{-|x|}^{|x|} \Theta_0 (x,y)
  \Lambda_\alpha\tau_y (\Lambda_\alpha f)(a) A_\alpha(y) dy\\
  &=&\Lambda_\alpha\big( I_1(x,\Lambda_\alpha f)\big)(a).\end{eqnarray*}
    Suppose that $$\Lambda_\alpha^{k+1}
    \big(I_k(x,f)\big)(a)=\Lambda_\alpha^{k}\big(I_k(x,\Lambda_\alpha f)\big)(a),$$
this gives \begin{align*}\Lambda_\alpha^{k+2}
\big(I_{k+1}(x,f)\big)(a)&= \int_{-|x|}^{|x|} \Theta_0 (x,y)
 \Lambda_\alpha \big(\Lambda_\alpha^{k+1}
    I_k(y,f)\big)(a) A_\alpha(y) dy\\&= \int_{-|x|}^{|x|} \Theta_0 (x,y)
 \Lambda_\alpha \big(\Lambda_\alpha^{k}
    I_k(y,\Lambda_\alpha f)\big)(a) A_\alpha(y) dy\\&= \Lambda_\alpha^{k+1}
\big(I_{k+1}(x,\Lambda_\alpha f)\big)(a).\end{align*}
 Then, we obtain (3.7) by induction.\\
$\bullet$ From (2.1) and (2.4), we can write
 \begin{align*}
   \Lambda_\alpha \big(I_1(x,f)\big)(a) &=    \int_{-|x|}^{|x|} \Theta_0 (x,y) \Lambda_\alpha(\tau_y f)(a) A_\alpha(y) dy\\
&=  \int_{-|x|}^{|x|} \Theta_0 (x,y)\tau_y (\Lambda_\alpha f)(a)A_\alpha(y) dy\\
&= R_1(x,f)(a).
   \end{align*}
    Suppose that
   $$\Lambda_\alpha^k\big( I_k(x,f)\big)(a) = R_k(x,f)(a),$$
    then by (3.6) and (3.7), we have
    \begin{align*}
     \Lambda_\alpha^{k+1} \big(I_{k+1}(x,f)\big)(a)&= \int_{-|x|}^{|x|} \Theta_0 (x,y) \Lambda_\alpha^{k+1}
      \big(I_k(y,f)\big)(a)A_\alpha(y) dy\\
     &=  \int_{-|x|}^{|x|} \Theta_0 (x,y) \Lambda_\alpha^k \big(I_k(y,\Lambda_\alpha f)\big)(a)A_\alpha(y) dy\\
     &=  \int_{-|x|}^{|x|} \Theta_0 (x,y) R_k(y,\Lambda_\alpha f)(a)A_\alpha(y) dy\\
     &= R_{k+1}(x,f)(a).
    \end{align*} By induction, this gives (3.8).
\end{proof}

Before establishing that
$\mathcal{B}^k\mathcal{D}_{p,q}^{\beta,\alpha}=\mathcal{K}^k\mathcal{D}_{p,q}^{\beta,\alpha},$
we give a remark, a proposition containing sufficient conditions and
an example.
\begin{rem} For $k=1, 2,...,$ $f\in \mathcal{E}(\mathbb{R})\cap L^p(\mu_\alpha)$ such that
$\Lambda_\alpha^{k-1}f$ is in $L^p(\mu_\alpha)$ and
$x\in(0,+\infty),$ we can assert from (3.1) that
\begin{enumerate}
  \item $\omega_{p,\alpha}^k(x,f)=\displaystyle\sup_{|y| \leq
x} \| R_{k}(y,f)\|_{p,\alpha}.$
\item For $k=1$, we have  $\omega_{p,\alpha}^k(x,f) =\displaystyle
\sup_{|y| \leq x} \| \tau_y(f)- f\|_{p,\alpha}.$
\end{enumerate}
\end{rem}
\begin{prop} Let $k=1, 2,...,$ $ 1 \leq p < +\infty $, $ 1 \leq q \leq +\infty$,
$0<\beta<1$ and $f\in \mathcal{E}(\mathbb{R})\cap L^p(\mu_\alpha).$
If $\Lambda_\alpha^{k-1}f$ and $\Lambda_\alpha^k f$ are in
$L^p(\mu_\alpha)$, then
$f\in\mathcal{B}^k\mathcal{D}_{p,q}^{\beta,\alpha}.$
\end{prop}
\begin{proof} Let $k=1, 2,...,$ $ 1 \leq p < +\infty $, $ 1 \leq q \leq +\infty$,
$0<\beta<1$ and $f\in \mathcal{E}(\mathbb{R})$ such that
$\Lambda_\alpha^{k-1}f,\,\Lambda_\alpha^k f$ are in
$L^p(\mu_\alpha).$ For $x \in (0,+\infty),$ we obtain by (3.3) and
(3.4),
\begin{align*}\omega_{p,\alpha}^k(x,f)\leq c \,x^k \lVert \Lambda_\alpha^k f
\rVert_{p,\alpha} \quad\mbox{and}\quad \omega_{p,\alpha}^k(x,f)\leq
c \,x^{k-1} \lVert \Lambda_\alpha^{k-1} f
\rVert_{p,\alpha}.\end{align*} Then we can write,
\begin{align*}
\int_0^{+\infty}
\Big(\frac{\omega_{p,\alpha}^k(x,f)}{x^{\beta+k-1}}\Big)^q
\frac{dx}{x}\leq  c \int_0^1 \Big(\frac{\| \Lambda_\alpha^k f
\|_{p,\alpha}}{x^{\beta-1}}\Big)^q \frac{dx}{x}+c \int_1^{+\infty}
\Big(\frac{\| \Lambda_\alpha^{k-1} f
\|_{p,\alpha}}{x^{\beta}}\Big)^q \frac{dx}{x},
\end{align*} giving two finite integrals. Here when $q=+\infty$, we make the usual
modification.
\end{proof}
\begin{ex} From (2.5) and Proposition 3.1, we can assert that the
spaces $\mathcal{C}_c^\infty(\mathbb{R})$ and
$\mathcal{S}(\mathbb{R})$ are included in
$\mathcal{B}^k\mathcal{D}_{p,q}^{\beta,\alpha}.$
\end{ex}
\begin{thm}
  Let $0 <\beta <1$, $ k=1,2,...,$ $ 1 \leq p < +\infty $ and $ 1 \leq q \leq +\infty$, then
  \begin{equation*}
   \mathcal{B}^k\mathcal{D}_{p,q}^{\beta,\alpha} =
   \mathcal{K}^k\mathcal{D}_{p,q}^{\beta,\alpha}.
  \end{equation*}
  \end{thm}
\begin{proof}
   We start with the proof of the inclusion
   $\mathcal{K}^k\mathcal{D}_{p,q}^{\beta,\alpha}\subset\mathcal{B}^k\mathcal{D}_{p,q}^{\beta,\alpha}$.
    Let $ f$ a function in $\mathcal{K}^k\mathcal{D}_{p,q}^{\beta,\alpha}$. If $f=f_0+f_1$, $f_0 \in \mathcal{D}_{p,\alpha}^{k-1}$
     and $f_1 \in \mathcal{D}_{p,\alpha}^k$ is any decomposition of
     $f,$ we have by (3.3)
 \begin{align}
   \omega_{p,\alpha}^k(x,f_1) &=  \sup_{|y| \leq x} \| R_k(y,f_1)\|_{p,\alpha}\nonumber\\
   &\leq c \  x^k \| \Lambda_\alpha^k f_1 \|_{p,\alpha},\quad  x
   \in (0,+\infty).
 \end{align}
Using (3.4), we obtain
  \begin{align}
   \omega_{p,\alpha}^k(x,f_0) &\leq \sup_{|y| \leq x} \| R_{k-1}(y,f_0)\|_{p,\alpha}+
    \sup_{|y| \leq x} \| b_{k-1}(y) \Lambda_\alpha^{k-1} f_0  \|_{p,\alpha} \nonumber \\
   &\leq c \  x^{k-1} \| \Lambda_\alpha^{k-1} f_0 \|_{p,\alpha}, \quad  x
   \in (0,+\infty).
  \end{align}
Hence by (3.9) et (3.10), we deduce that
\begin{align*}
 \omega_{p,\alpha}^k(x,f) \leq c \  x^{k-1} K_{p,\alpha}^k(x,f),
\end{align*}
then, $ f \in \mathcal{B}^k\mathcal{D}_{p,q}^{\beta,\alpha}.$\\
  To prove the inclusion $\mathcal{B}^k\mathcal{D}_{p,q}^{\beta,\alpha}
  \subset \mathcal{K}^k\mathcal{D}_{p,q}^{\beta,\alpha}$, we have to make for
  $ f \in \mathcal{B}^k\mathcal{D}_{p,q}^{\beta,\alpha}$ a proper choice of $f_0$ and $f_1.$ We
  take for $ x\in (0,+\infty)$
  $$ f_1 = \frac{1}{b_k(x)}  \;  I_k(x,f). $$
  Using (3.8), we obtain
\begin{align}
  x \| \Lambda_\alpha^k f_1 \|_{p,\alpha} &\leq \ x \; \big(b_k(x)\big)^{-1}\omega_{p,\alpha}^k(x,f)\nonumber\\
&\leq c \; \frac{\omega_{p,\alpha}^k(x,f)}{x^{k-1}}.
  \end{align}
On the other hand, put $f_0= f-f_1 $, we can write by (3.5)
$$ f_0 = -\frac{1}{b_k(x)}  \int_{-x}^{x} \Theta_0 (x,y)
 \big(I_{k-1}(y,f) - b_{k-1}(y) f \big) A_\alpha(y) dy.$$
  From (3.1) and (3.8), we obtain
  $$ \Lambda_\alpha^{k-1} f_0 = -\frac{1}{b_k(x)} \int_{-x}^{x} \Theta_0 (x,y)R_k(y,f)A_\alpha(y) dy.$$
By Minkowski's inequality for integrals and (3.2), we get
\begin{align}
  \| \Lambda_\alpha^{k-1} f_0 \|_{p,\alpha}
&\leq  \big(b_k(x)\big)^{-1} \int_{-x}^{x} |\Theta_0 (x,y)| \ \|
R_k(y,f)
 \|_{p,\alpha} A_\alpha(y) dy \nonumber \\
  & \leq  c \; x^{-k} \omega_{p,\alpha}^k(x,f)\int_{-x}^{x} |\Theta_0 (x,y)| \ A_\alpha(y) dy \nonumber \\
 &\leq  c \; \frac{\omega_{p,\alpha}^k(x,f)}{x^{k-1}}.
\end{align}
By (3.11) et (3.12), we deduce that
\begin{align*}
 K_{p,\alpha}^k(x,f) \leq c\
 \frac{\omega_{p,\alpha}^k(x,f)}{x^{k-1}},
\end{align*}
then, $f \in \mathcal{K}^k\mathcal{D}_{p,q}^{\beta,\alpha}$ which
completes the proof of the theorem.
 \end{proof}

\end{document}